\documentclass[12pt,a4paper]{amsart}
\usepackage[utf8]{inputenc} %

\title{Drawing strategies in Strong Ramsey games for 3-uniform hypergraphs}
\author[N. Bowler]{Nathan Bowler}
\author[H. Ortm{\"u}ller]{Henri Ortm{\"u}ller}
\email{nathan.bowler@uni-hamburg.de, henri.ortmueller@unifr.ch}

\date{\today}

\usepackage{graphicx} %

\usepackage[left=25mm,right=25mm,top=30mm,bottom=30mm,includeheadfoot]{geometry} %
\usepackage{amsmath,amsthm,amssymb} %
\usepackage{thmtools} %
\usepackage{mathtools} %
\usepackage{dsfont} %
\usepackage[colorlinks, linkcolor = magenta]{hyperref} %
\usepackage[nameinlink]{cleveref}
\usepackage{caption} %
\usepackage{subcaption} %
\usepackage[backend=biber]{biblatex} %
\usepackage{tikz, tikz-network} %
\usepackage[many]{tcolorbox} %
\usepackage{enumitem} %
\setlist[enumerate]{label=(\roman*)}
\usepackage{float} %
\allowdisplaybreaks %
\usepackage[makeroom]{cancel}

\newcommand{\R}{\mathcal{R}} %

\newtheorem{theorem}{Theorem}[section] 
\newtheorem{lemma}[theorem]{Lemma}

\newtheorem{corollary}[theorem]{Corollary}
\theoremstyle{definition}
\newtheorem{definition}[theorem]{Definition}

\newtheorem{question}[theorem]{Open problem}

\newtcolorbox{claimproof}{%
    enhanced,
    colback=white,
    colframe=white,
    left=5pt,
    top=5pt,
    right=5pt,
    bottom=5pt,
    sharp corners,
    breakable,
    overlay first={\draw[thick] 
        ([yshift=-2em]frame.north west) |- ([xshift=2em]frame.north west)
        ;},
    overlay last={\draw[thick] 
        ([yshift=2em]frame.south east) |- ([xshift=-2em]frame.south east)
        ;},
    overlay unbroken={\draw[thick] 
        ([yshift=-2em]frame.north west) |- ([xshift=2em]frame.north west)
        ([yshift=2em]frame.south east) |- ([xshift=-2em]frame.south east)
        ;}
}

\begin{document}

\begin{abstract}
    The \emph{Strong Ramsey game} $\mathcal{R}(B,G)$ is a two player game with players $P_1$ and $P_2$, where $B$ and $G$ are $k$-uniform hypergraphs for some $k \geq 2$. $G$ is always finite, while $B$ may be infinite. $P_1$ and $P_2$ alternately color uncolored edges $e \in B$ in their respective color and $P_1$ begins. Whoever completes a monochromatic copy of $G$ in their own color first, wins the game. If no one claims a monochromatic copy of $G$ in a finite number of moves, the game is declared a draw. In this paper, we give an infinite family of 3-uniform hypergraphs $\{G_t\}_{t \geq 3}$, such that $P_2$ has a drawing strategy in the Strong Ramsey game $\mathcal{R}(K_{\aleph_0}^{(3)}, G_t)$. This improves a result by David, Hartarsky and Tiba \cite{MR4073378}.
\end{abstract}

\maketitle

\section{Introduction}
One of the first people to study the Strong Ramsey game was Beck. In \cite{BeckRamseyGames2002}, he raised the question, whether $P_1$ has a winning strategy in $\R(K_{\aleph_0}, K_k)$ for all $k \geq 4$ and $\aleph_0$ being the smallest infinite cardinal. For $k=4$, his paper contains a winning strategy for $P_1$, which later turned out to be incomplete. Bowler and Gut \cite{bowler2023k4game} fixed the argument by using a computer algorithm. Note that $P_2$ does not have a winning strategy in $\R(B,G)$ for any graphs $B,G$, due to a simple strategy stealing trick (see \cite[Theorem 1.3.1]{MR3524719}).

In the meantime, Hefetz, Kusch, Narins, Pokrovskiy, Requil\'e and Sarid \cite{MR3645576} constructed a $5$-uniform hypergraph $G^*$ for which they proved that $\R(K^{(5)}_{\aleph_0},G^*)$ is a draw. Moreover, they suggested to tackle the following problem as an intermediate step towards Beck's question:
\begin{question}
    Does there exist a graph $G$, such that $\R(K_{\aleph_0},G)$ is a draw?
\end{question}
For $l \geq 3$, let $\hat{K}_{2,l}$ be the graph obtained from $K_{2,l}$ together with the edge connecting the two vertices of degree $l$. For a graph $G$, let $G^{(k)}$ denote the $k$-uniform hypergraph, which can be obtained by adding the same fixed set of $k-2$ vertices to every edge in $E(G)$.
In 2020, David, Hartarsky and Tiba \cite{MR4073378} improved the above result by proving that $P_2$ has a drawing strategy in $\R(K_{\aleph_0}^{(4)},\hat{K}_{2,4}^{(4)})$. In \Cref{sec:K24+} of this paper, we improve their result even further. Formally, we show that $P_2$ has a drawing strategy in $\R(K_{\aleph_0}^{(3)},\hat{K}_{2,4}^{(3)})$, thereby giving the first example of a 3-uniform hypergraph for which the corresponding Strong Ramsey game is a draw. This also supports their conjecture that $\R(K_{\aleph_0},\hat{K}_{2,4})$ is a draw. Moreover, we proved in \cite{bowlerOrt2025} that, if $\R(K_{\aleph_0},\hat{K}_{2,4})$ were a draw, it would be a minimal example of this type.

For $s \geq 0$, $t \geq 3$, let $K_{2,t}(s)$ be the graph obtained by identifying the center of an $s$-star with one of the vertices of degree $t$ of $K_{2,t}$. Recently, Ai, Gao, Xu and Yan \cite{ai2025strongramseygameboards} proved that $P_2$ has a drawing strategy in the Strong Ramsey game played on two disjoint boards $\R(K_{\aleph_0} \sqcup K_{\aleph_0}, K_{2,t+1}(t-2))$ for all $t \geq 3$.
Building on their work, we prove in \Cref{sec:K2t} that $\R(K^{(3)}_{\aleph_0}, K^{(3)}_{2,t+1}(t-2))$ is a draw for all $t \geq 3$. This suggests that the graphs $K_{2,t+1}(t-2)$ are promising candidates for a draw in the game $\R(K_{\aleph_0}, K_{2,t+1}(t-2))$.

\begin{figure}[H]
    \centering
    \begin{subfigure}[t]{0.23\textwidth}
        \centering
        \includegraphics[height = 28mm, keepaspectratio]{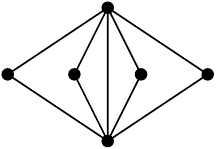}
        \caption{The graph $\hat{K}_{2,4}$}
    \end{subfigure}
    \hspace{0.15\textwidth}
    \begin{subfigure}[t]{0.23\textwidth}
        \centering
        \includegraphics[height= 28mm, keepaspectratio]{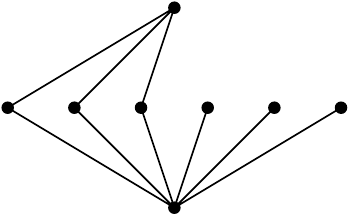}
        \caption{The graph $K_{2,3}(3)$}
    \end{subfigure}
    \caption{}
\end{figure}

\section{Notation, Definitions and Basic Results}
All (hyper)graphs considered in this paper are undirected and simple. Hence, we denote $k$-(hyper-)edges as $x_1x_2 \ldots x_k$ instead of $\{x_1, x_2, \ldots, x_k \}$ for clarity of the presentation.
In the Strong Ramsey game $\R(B,G)$, we will call $B$ the \textit{board} and $G$ the \textit{target graph}. We refer to $P_1$ as \textit{she} and $P_2$ as \textit{he} and we will often say that a player claims, picks or takes an edge instead of colors it in their respective color. For brevity, we sometimes refer to an edge claimed by $P_1$ or $P_2$ as $P_1$-edge or $P_2$-edge, respectively. We use a similar notion for graphs which a player has claimed. We might also say that a player has or has built a graph $H$, meaning that the player has claimed all edges of a copy of $H \subseteq B$. Moreover, we say a vertex $x \in B$ is fresh, if no player has claimed an edge incident to $x$.
For a given target graph $G$, we say that $P_1$ has a \textit{threat} or $P_1$ \textit{threatens} $P_2$, if $P_1$ has claimed a copy $H \subseteq B$ of $G-e$ for an edge $e \in E(G)$, such that the corresponding $\hat{e} \in B$ is not yet claimed by either player, and similarly for $P_2$. We will call $H$ a \textit{threat graph} in that case.
In order to be concise, we will not give a formal definition of \emph{strategy} here. Instead, we refer the reader to \cite[Appendix C]{CombGamesBeck}. We define $[k] \coloneqq \{1, \dots, k\}$.

We call the vertices added in order to obtain $G^{(k)}$ from $G$ the \emph{centers} of $G^{(k)}$.
On the other hand, in the case of $G = K_{\aleph_0}$ and $k=3$, given a vertex $x \in V(K_{\aleph_0}^{(3)})$, observe that the set $\{xyz \ | \ y,z \in V(K_{\aleph_0}^{(3)})\}$ spans $E(K_{\aleph_0})$. In this case, we will also refer to $x$ as \textit{center} and we will call $\{yz \ | \ xyz \in E(K_{\aleph_0}^{(3)})\}$ the $x$-\textit{board}. In both $K_{2,t}(s)$ and $\hat{K}_{2,l}$, all vertices of degree at least 3 will be called \emph{main vertices} and all vertices of degree at most 2, \emph{minor vertices}. Moreover in $\hat{K}_{2,l}$, the edge connecting the two vertices of degree $l+1$ is called the \textit{main edge}.

In the course of this paper, it will be convenient to use time stamps in order to clarify which stage of the game we are talking about. Therefore, for a fixed point in time $T$,
we define $E_T(P_1)$ as the set of edges and $e_T(P_1)$ as the number of edges that $P_1$ has claimed at $T$. For vertices $v,w \in B$, we let $d_T^{P_1}(v)$ be the number of edges from $E_T(P_1)$, which are incident to $v$.

In the proof of \Cref{thm:K24+} and \Cref{thm:K2t^3}, $P_2$ must react appropriately in order to prevent $P_1$ from claiming a copy of $G$. It will turn out to be helpful to distinguish cases based on how much progress $P_1$ has made toward claiming such a copy of $G$. For this, we will introduce the following \emph{progress bar}.
\begin{definition}
    \label{def:eTP1(G)}
    For a given target graph $G^*$ and board $B$ at a given point in time $T$, we define
    \begin{equation*}
        e_T^{P_1}(G^*) = \max_{\mathclap{\substack{H \subseteq B \\ H \cong G^* \\ E(H) \cap E_T(P_2) = \emptyset}}}\{|E_T(P_1) \cap E(H)|\}.
    \end{equation*}
\end{definition}
From the point in time $T$ onward, notice that it takes $P_1$ at least $e(G^*) - e_T^{P_1}(G^*)$ moves to claim a copy of $G^*$. In particular, she has a threat at $T$ if and only if $e(G^*) - e_T^{P_1}(G^*) = 1$. All the expressions above are defined similarly for $P_2$, and we might omit $T$ in case it is clear from the context.

In this paper, we are going to use a particular kind of figure to depict some board states in order to give the reader a visual intuition. In these figures, we will use the color violet only for edges of $P_1$ and blue only for edges of $P_2$. Moreover, all figures in this paper depict a board state after a move of $P_1$.
Since drawing figures of 3-uniform hypergraphs is often a nightmare, we will only depict $x$-boards for specific $x \in B$, which we will indicate at the top of the figure. Hence, all the edges that are drawn as $2$-uniform edges correspond to 3-uniform edges containing the center of the corresponding board.
In case, $P_1$ has claimed an edge $z_1z_2z_3$ not containing $x$, but the $x$-board is depicted in the figure, we will denote this by $+z_1z_2z_3$ in the right corner of the figure. We will also use this notation, if $P_1$ has claimed an edge $e^*$ with special properties. Say that $P_1$ has claimed $k$ additional edges for which we do not have any further information. We will denote this by $+k$. Note that, we will not indicate, if edges of $P_2$ are missing in the figure.
Oftentimes, we add a label $T$ to the top left corner of a figure in order to indicate which point in time it depicts.

For example in \Cref{pic:K24+:T1}, the $a$-board is depicted at $T_1$, $P_1$ has claimed the edges $z_1z_2z_3$, $e^*$ with special properties and four additional edges that we do not have any information about whereas $P_2$ has claimed $abv_1, abv_2, abv_3, abv_4, abv_5$ so far.

\section{\texorpdfstring{$\R(K_{\aleph_0}^{(3)}, \hat{K}_{2,4}^{(3)})$}{The 3-uniform K24-game} is a draw}
\label{sec:K24+}

In this section, all we do is proving that $\R(K_{\aleph_0}^{(3)}, \hat{K}_{2,4}^{(3)})$ is a draw. For convenience, we let $G = \hat{K}_{2,4}^{(3)}$ in this section and we will refer to $G^-$ as the subgraph of $G$ with 8 edges, that is not the $K_{2,4}^{(3)}$.

\begin{theorem}
    \label{thm:K24+}
    $\R(K_{\aleph_0}^{(3)}, G)$ is a draw.
\end{theorem}
In order to prove a draw in a Strong Ramsey game, from a certain board state onward, $P_2$'s strategy has to distract $P_1$ constantly from claiming a copy of $G$. This is only possible if $P_2$ threatens to claim a copy of $G$ himself. In the following lemma, we present a type of board state for which $P_2$ can make use of such a distraction strategy in order to force at least a draw. The idea of distracting $P_1$ was first used in \cite{MR3645576} and it is essential for the proof of \Cref{thm:K24+}, since we only have to check the conditions required by this lemma, instead of playing the game for a countable number of moves, in order to guarantee a draw for $P_2$.
Note that in the proof of \Cref{thm:K24+}, we always make use of the Distraction Lemma (\ref{lem:K24+:EndPosition}), whenever we state that $P_2$ has a drawing strategy from a given board state onward.
\begin{lemma}[Distraction Lemma]
    \label{lem:K24+:EndPosition}
    Suppose $P_2$ has claimed a $\hat{K}_{2,3}^{(3)}$ with main vertices $x,y$ and center $c$ in the game $\R(K_{\aleph_0}^{(3)}, G)$, such that
    \begin{enumerate}
        \item  $P_1$ does not have a threat;
        \item  $P_1$ does not have a $\hat{K}_{2,3}^{(3)}$ with center $c$ and main vertex $x$;
        \item  $P_1$ does not have a $\hat{K}_{2,3}^{(3)}$ with center $x$ and main vertex $c$.
    \end{enumerate}
    If it is $P_2$'s turn, he has a drawing strategy.
\end{lemma}
\begin{proof}
    Let $a_1$ be a fresh vertex. $P_2$ claims $cya_1$ and $P_1$ must block $P_2$'s threat by taking $cxa_1$, because she had no threat. If claiming $cxa_1$ provides her with a threat graph $H$, we have $cxa_1 \in E(H)$.
    Note that every edge of $H$ is incident to the center and one of the main vertices. Since $d^{P_1}(a_1) = 1$, it can neither be a center nor a main vertex of $H$. So $c$ and $x$ must be the center and one of the main vertices of $H$, but every threat graph with a vertex of degree 1 contains a $\hat{K}_{2,3}^{(3)}$, which was excluded by the assumptions (ii) and (iii).

    $P_2$ continues by claiming $cya_i$ for fresh vertices $a_i$ as long as the game does not end and he cannot win by claiming $cxa_i$. Assume for a contradiction that $P_1$ won the game eventually. Then, she must have obtained a threat graph $H$ by blocking with $cxa_i$. Using the same
    argument as above, $P_1$ must have claimed a copy of $\hat{K}_{2,3}^{(3)}$ of type (ii) or (iii), which was forbidden. Hence, $P_2$ has a drawing strategy.
\end{proof}

\begin{corollary}
    \label{cor:K24+:EndPosition}
    In the game $\R(K_{\aleph_0}^{(3)}, G)$, suppose $P_2$ has built a $\hat{K}_{2,3}^{(3)}$, such that $P_1$ has claimed at most 10 edges and she has no threat.
    If it is $P_2$'s turn, he has a drawing strategy.
\end{corollary}
\begin{proof}
    Let $x,y$ be the main vertices and $c$ the center of the $\hat{K}_{2,3}^{(3)}$, which $P_2$ has built and assume that either (ii) or (iii) of the Distraction Lemma (\ref{lem:K24+:EndPosition}) does not hold.
    In this case, $P_1$ has claimed at least seven edges not incident to $y$, since $cxy \in E(P_2)$. $P_1$ cannot have claimed another $\hat{K}_{2,3}^{(3)}$ with either main vertex or center $y$, since this requires at least four more edges on the $y$-board. Hence, $P_1$ has not claimed a copy of $\hat{K}_{2,3}^{(3)}$ with center $c$ or $y$ and main vertex $y$ or $c$, respectively. If we exchange $x$ and $y$ in the Distraction Lemma (\ref{lem:K24+:EndPosition}), we obtain that $P_2$ can force at least a draw.
\end{proof}
With those auxiliary results, we are finally able to prove the main theorem of this section.

\begin{proof}[Proof of \Cref{thm:K24+}]
    $P_1$ claims an edge $z_1z_2z_3$. In his first five moves $P_2$ takes $abv_i$ for fresh vertices $a$, $b$, $v_i$ with $i \in [5]$. Let $e^*$ be the second edge of $P_1$ and let $T_1$ be the point in time after her sixth move. We obtain the board state in \Cref{pic:K24+:T1}.
    \begin{figure}[H]
        \centering
        \includegraphics[width=0.30\textwidth]{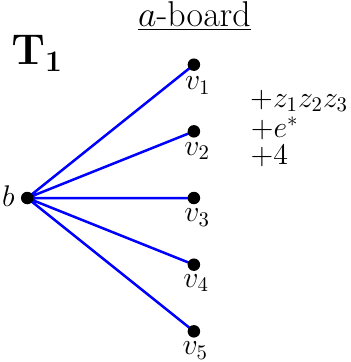}
        \caption{The $a$-board at $T_1$. Note that the $b$-board looks almost identical: We only need to exchange $a$ and $b$ in the figure.}
        \label{pic:K24+:T1}
    \end{figure}
    For a vertex $v \in V(B) \backslash  \{v_1, \dots, v_5\}$, define
    \begin{align*}
        F_v \coloneqq \{v_i \in \{v_1, \dots, v_5\} \ | \ \forall j \neq i \colon vv_iv_j \notin E_{T_1}(P_1)\}
    \end{align*}
    as set of vertices, which are not incident to a violet edge in the $K_5$ induced by $\{v_1, \dots, v_5\}$ on the $v$-board. At $T_1$, we are in one of the following two cases.

    \underline{Case 1}: $|F_a| \geq 2$ or $|F_b| \geq 2$.  \newline
    \underline{Case 2}: $P_1$ has claimed the edge set $\{z_1z_2z_3, e^*, av_{i_1}v_{i_2}, av_{i_3}v_{i_4}, bv_{j_1}v_{j_2}, bv_{j_3}v_{j_4}\}$ for pairwise different $i_1,i_2,i_3,i_4 \in [5]$ and $j_1,j_2,j_3,j_4 \in [5]$. \newline

    We first show that this case distinction is complete. Suppose $|F_a| \leq 1$ and $|F_b| \leq 1$. By choice of the $v_i$, $e^*$ is incident to at most one $v_i$. Since $P_1$ has only claimed six edges at $T_1$, the four edges of $P_1$ not specified yet must be of the form $bv_{j_1}v_{j_2}, bv_{j_3}v_{j_4}$ and $av_{i_1}v_{i_2}, av_{i_3}v_{i_4}$ in order to fulfill $|F_a| \leq 1$ and $|F_b| \leq 1$. Hence, we are \textit{Case 2}.

    We start with proving that $P_2$ can force at least a draw in the first case.

    \begin{proof}[Proof of Case 1]
        In this case, $P_2$ will build a threat graph; its center $c$ will be chosen as follows. Let $c \in \{a,b\}$ be the vertex at $T_1$, such that $|F_c| \geq 2$. If both $|F_a| \geq 2$ and $|F_b| \geq 2$, we choose $c$, such that either $c \notin e^*$ or both $a,b \in e^*$.
        We let $x \in \{a,b\}$ be the other vertex, which will become a main vertex of the threat graph $P_2$ builds.

        $P_2$ claims $cv_iv_j$, such that $i,j \in F_c$ and  $d^{P_1}_{T_1}(v_i), d^{P_1}_{T_1}(v_j)$ are minimal.
        Up to relabeling, suppose $P_2$ took $cv_4v_5$ in his sixth move. Moreover, we can assume up to relabeling that $cv_1v_5, cv_2v_5, cv_3v_5 \notin E(P_1)$ after the seventh move of $P_1$. In his seventh move, $P_2$ takes $cv_kv_5$ for $k \in [3]$, such that $d^{P_1}(v_k)$ is currently maximal.
        Again by relabeling, suppose $P_2$ has taken $cv_3v_5$. Let $T_2$ be the point in time after the eighth move of $P_1$ (see \Cref{pic:K24:T2}). Note that one of $cv_1v_5$ and $cv_2v_5$ is not claimed by $P_1$ at $T_2$.
        \begin{figure}[H]
            \centering
            \includegraphics[width=0.27\textwidth]{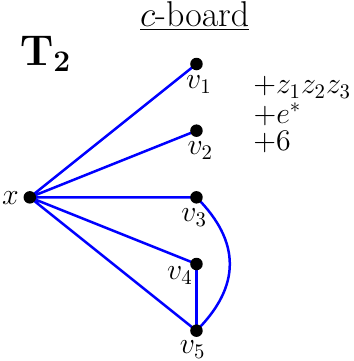}
            \caption{The $c$-board at $T_2$.}
            \label{pic:K24:T2}
        \end{figure}
        We distinguish between the following four cases.

        \underline{Case 1.1}: $e_{T_2}^{P_1}(G) \leq 6$.
        \begin{claimproof}
            \label{pf:K24+:Case1.1}
            By symmetry, assume that $cv_2v_5$ is not taken by $P_1$, so $P_2$ claims $cv_2v_5$ and therefore builds a $\hat{K}_{2,3}^{(3)}$. After her next move, $P_1$ cannot have a threat by assumption and since $P_1$ has claimed 9 edges so far, $P_2$ can force at least a draw by \Cref{cor:K24+:EndPosition}.
        \end{claimproof}
        \underline{Case 1.2}: $e_{T_2}^{P_1}(G) = 8$ and $P_1$ has built a copy of $K_{2,4}^{(3)}$.
        \begin{claimproof}
            For simplicity, let $H$ be the copy of $K_{2,4}^{(3)}$, which $P_1$ built. Since $z_1z_2z_3 \in E(H)$ and $e_{T_2}^{P_1}(G) = 8$, observe that both $cv_1v_5$ and $cv_2v_5$ are unclaimed at $T_2$. $P_2$ begins by blocking the threat of $P_1$. Immediately after his move, one can observe that $e^{P_1}(G) = 4$ (by checking every combination of centers and main vertices). So as in \emph{Case 1.1}, by symmetry $P_2$ picks $cv_2v_5$ in his next move, thereby claiming a $\hat{K}_{2,3}^{(3)}$. It is easy to see that $P_1$ cannot have a threat after her tenth move, since $e^{P_1}(G) = 4$ two moves ago. By \Cref{cor:K24+:EndPosition}, $P_2$ obtains a drawing strategy.
        \end{claimproof}

        \underline{Case 1.3}: $e_{T_2}^{P_1}(G) = 8$ and $P_1$ has built a copy of $G^-$.
        \begin{claimproof}
            Let $H$ be the copy of $G^-$ claimed by $P_1$. As in \emph{Case 1.2}, we have $z_1z_2z_3 \in E(H)$. So without loss of generality, let $z_1$ be the center of $H$, let $z_2$ be a main vertex and let $z^*$ be the other main vertex. Both $cv_1v_5$ and $cv_2v_5$ are not taken by either player. \newline

            $P_2$ blocks the threat of $P_1$. As long as $P_1$ creates threat graphs by claiming either $z_1z_2w_i$ or $z_1z^*w_i$ for vertices $w_i$, $P_2$ blocks by taking $z_1z^*w_i$ or $z_1z_2w_i$, respectively. Since the game ends in a draw, if she does not stop claiming edges of those type, eventually $P_1$ must claim an edge $\hat{e}$, which does not contain both $z_1$ and $z_2$ or both $z_1$ and $z^*$. Let $T_3$ be the point in time after $P_1$ has claimed $\hat{e}$. We obtain the board state in \Cref{pic:K24+:T3} (from a $P_1$ perspective), where the left blob represents the vertices $v$ for which $P_1$ picked $z_1z_2v$ and $P_2$ blocked accordingly and the right blob represents all $w$ for which $P_1$ claimed $z_1z^*w$ and $P_2$ blocked accordingly.
            \begin{figure}[H]
                \centering
                \includegraphics[width=0.32\textwidth]{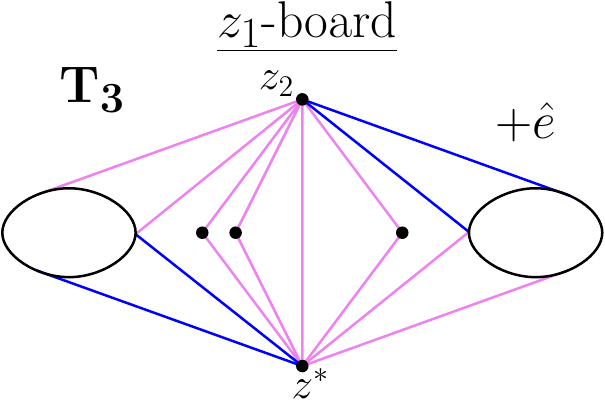}
                \caption{The $z_1$-board with all edges of $P_1$ after she took $\hat{e}$.}
                \label{pic:K24+:T3}
            \end{figure}
            We will now prove that $e_{T_3}^{P_1}(G) = 7$ and hence she has no threat. Assume for contradiction that after $P_1$ took $\hat{e}$, she has claimed a threat graph $H$. Then, $\hat{e} \in E(H)$. Additionally, the center and the main vertices of $H$ must have degree at least 4 and are therefore $z_1, z_2$ and $z^*$. If $z_1$ was the center, then $z_2$ and $z^*$ have to be the main vertices of $H$. In this case, $\hat{e}$ is incident to either both $z_1$ and $z_2$ or both $z_1$ and $z^*$, but those possibilities were excluded. $z_2$ cannot be the center, because the $z_2$-board contains at most a $P_1$-star centered at $z_1$ and potentially $\hat{e}$. Hence, there cannot be a threat with center $z_2$, and similarly for $z^*$. \newline

            $P_2$ claims $cv_2v_5$ up to relabeling vertices. This yields him a $\hat{K}_{2,3}^{(3)}$ with main vertices $x$ and $v_5$. We have either $\hat{e} = cv_1v_5$ or $\hat{e} \neq cv_1v_5$. In the latter case, $P_1$ must have taken $cv_1v_5$ in her next move after $T_3$, because he wins on the next move otherwise. \newline

            If $\hat{e} \neq cv_1v_5$, let $T_4$ be the point in time after $P_1$ took $cv_1v_5$. We check the conditions of the Distraction Lemma (\ref{lem:K24+:EndPosition}) at $T_4$.
            Suppose $P_1$ has claimed a copy of a threat graph $H$.
            Since she had no threat a move ago (at $T_3$), we must have $cv_1v_5 \in E(H)$. Since the only possible centers for $H$ are $z_1, z_2$ and $z^*$, by choice of $c, v_1, v_5$, $z^*$ must be the center of $H$. This is not possible as the $z^*$-board contains at most a $P_1$-star centered at $z_1$ and two additional edges of $P_1$.
            For conditions (ii) and (iii), observe that if $z^* \neq x$, then $x$ is contained in at most three edges of $P_1$. Hence, it can neither be a main vertex nor a center of a $\hat{K}_{2,3}^{(3)}$. If $z^* = x$, then $c$ and $v_5$ satisfy conditions (ii) and (iii) as the only possible centers for a $\hat{K}_{2,3}^{(3)}$ claimed by $P_1$ are $z_1, z_2$ and $z^*$. \newline

            If $\hat{e} = cv_1v_5$, $P_1$ can begin a second threat phase, if she picks $z_1z_2w$ or $z_1z^*w$ for a vertex $w$. $P_2$ blocks accordingly as before until $P_1$ selects an edge, which does not contain both $z_1$ and $z_2$ or both $z_1$ and $z^*$. This results in the board state from the last paragraph at $T_4$ up to the number of times, that $P_1$ threatened $P_2$ by taking $z_1z_2w$ or $z_1z^*w$, respectively. Since $P_2$ had a drawing strategy in this position, this completes \textit{Case 1.3}.
        \end{claimproof}
        \underline{Case 1.4}: $e_{T_2}^{P_1}(G) = 7$.
        \begin{claimproof}
            $P_2$ claims $cv_2v_5$ up to relabeling. Since $P_1$ had no threat, we know that $cv_1v_5 \in E(P_1)$ after her ninth move. If $e^{P_1}(G) \leq 7$, we get by \Cref{cor:K24+:EndPosition} that $P_2$ can force at least a draw. So suppose that $e^{P_1}(G) = 8$ after her ninth move and let $H$ be the threat graph of $P_1$. First of all, $P_2$ blocks the threat. Since $\{z_1, z_2, z_3\} \cap \{c, v_1, v_5\} = \emptyset$, only one of those edges is contained in $E(H)$.
            \newline

            If $z_1z_2z_3 \in E(H)$, we obtain a board state, which can be reached via \textit{Case 1.2} or \textit{Case 1.3} as follows. Suppose $H$ was a copy of $G^-$, then $P_1$ could have reached the same position by claiming $G^-$ in her first eight moves. After $P_2$ blocks her threat as in \textit{Case 1.3}, $P_1$ takes $cv_1v_5$ and $P_2$ takes $cv_2v_5$. Similarly for \textit{Case 1.2}, if $P_1$ built a copy of $K_{2,4}^{(3)}$ within her first eight moves instead.  \newline

            Assume for contradiction that $cv_1v_5 \in E(H)$.
            If $c$ was the center of $H$, then $d_{T_1}^{P_1}(c) = 5$. If $|F_x| \leq 1$, this is not possible as $P_1$ has taken $z_1z_2z_3$ and $xv_iv_j$, which are both not on the $c$-board.  Otherwise, $|F_x| \geq 2$ and recall that at the beginning of this proof, we have ensured that $e^* = cxw$ for a vertex $w$ in this case. We obtain $e^* \notin E(H)$, since either $v_1$ or $v_5$ has to be a main vertex of $H$, but both $cxv_1$ and $cxv_5$ are taken by $P_2$.

            If $v_5$ was the center of $H$, then $d_{T_1}^{P_1}(v_5) = 5$ by construction. In case $|F_c| < 5$, we had $cv_{k_1}v_{k_2} \in E_{T_1}(P_1)$ with $k_1 \neq 5 \neq k_2$ and therefore $d_{T_1}^{P_1}(v_5) < 5$. If otherwise $|F_c| = 5$ and $d_{T_1}^{P_1}(v_5) = 5$, by the pigeonhole principle there exist $i,j \in [4]$, such that $d_{T_1}^{P_1}(v_i),d_{T_1}^{P_1}(v_j) \leq 3$ contradicting the choice of $cv_4v_5$ and hence $v_5$ cannot be a center of $H$.

            If $v_1$ was the center of $H$, after $P_1$'s seventh move, we must have had $d^{P_1}(v_1) = 6$. Since every vertex of $G$ other than the center has degree less than five, this implies $d^{P_1}(v_2), d^{P_1}(v_3) \leq 5$, which contradicts the choice of $cv_3v_5$.
        \end{claimproof}
        This completes the proof of \textit{Case 1}.
    \end{proof}
    It will be easier to show that $P_2$ can achieve at least a draw in the second case as it is more specific.
    \begin{proof}[Proof of Case 2]
        Suppose, we are at $T_1$ and $P_1$ has claimed the edges listed above. It is easy to check that $e_{T_1}^{P_1}(G) \leq 3$. Up to relabeling vertices, we can assume that $P_1$ has claimed $av_1v_2$ and $av_3v_4$ as in \Cref{pic:K24+:T5}.
        \begin{figure}[H]
            \centering
            \includegraphics[width=0.25\textwidth]{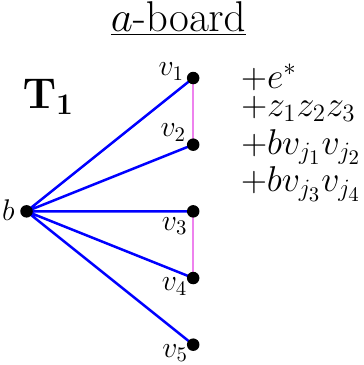}
            \caption{The $a$-board at $T_1$ in \textit{Case 2}.}
            \label{pic:K24+:T5}
        \end{figure}

        Without loss of generality, $P_2$ claims $av_4v_5$ and $av_1v_5$ in his next two moves. Thereafter, $P_2$ claims either $av_2v_5$ or $av_3v_5$, if possible (see \Cref{pic:K24+:Case2.1}) or both $av_2v_5$ and $av_3v_5$ were claimed by $P_1$. In the latter case, $P_2$ claims $av_1v_4$ and then $P_1$ did not claim both $av_1v_3$ and $av_2v_4$ after her next move. Up to symmetry, $P_2$ claims $av_2v_4$ in his next move to obtain a $\hat{K}_{2,3}^{(3)}$. This yields the board state in \Cref{pic:K24+:Case2.2}.
        \begin{figure}[H]
            \centering
            \begin{subfigure}[t]{0.33\textwidth}
                \centering
                \includegraphics[height= 40mm, keepaspectratio]{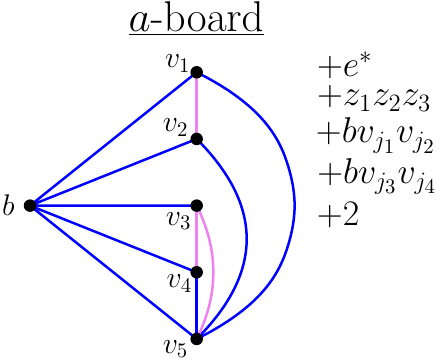}
                \caption{$P_1$ must have taken $av_3v_5$, because $P_2$ wins on the next move otherwise.}
                \label{pic:K24+:Case2.1}
            \end{subfigure}
            \hspace{0.15\textwidth}
            \begin{subfigure}[t]{0.33\textwidth}
                \centering
                \includegraphics[height= 40mm, keepaspectratio]{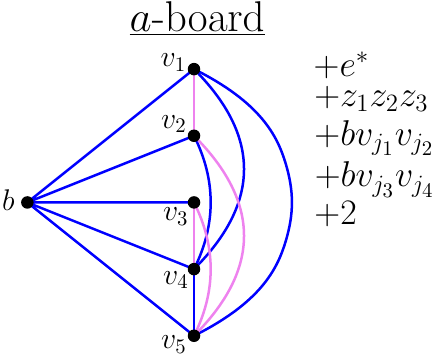}
                \caption{}
                \label{pic:K24+:Case2.2}
            \end{subfigure}
            \caption{}
        \end{figure}
        In \Cref{pic:K24+:Case2.1}, one can check that for all $x \in V(B)$, the $x$-board contains at most six edges of $P_1$ and therefore, we have $e^{P_1}(G) \leq 6$, so $P_2$ has a drawing strategy by \Cref{cor:K24+:EndPosition}.
        From comparing the edge sets $P_1$ has claimed in \Cref{pic:K24+:Case2.1} and \Cref{pic:K24+:Case2.2}, one can easily deduce that $e^{P_1}(G) \leq 7$ in \Cref{pic:K24+:Case2.2}. Hence by \Cref{cor:K24+:EndPosition}, $P_2$ has a drawing strategy in \textit{Case 2}.
    \end{proof}
    Since $P_2$ has a drawing strategy in \textit{Case 1} and \textit{Case 2} and those were all the cases that could occur, this completes the proof.
\end{proof}

\section{\texorpdfstring{$\R(K_{\aleph_0}^{(3)}, K^{(3)}_{2,t+1}(t-2))$}{The K2t-game} is a draw for all \texorpdfstring{$t \geq 3$}{t at least 3}}
\label{sec:K2t}
In this section, we are going to refine the main proof from Ai, Gao, Xu and Yan \cite{ai2025strongramseygameboards} in order to show that $\R(K^{(3)}_{\aleph_0}, K^{(3)}_{2,t+1}(t-2))$ is a draw for all $t \geq 3$. For convenience, we let $G_t = K^{(3)}_{2,t+1}(t-2)$ for the rest of this section. We begin by stating a lemma equivalent to the Distraction Lemma from \Cref{sec:K24+}. Due to their similarity, we will not prove it.

\begin{lemma}[Distraction Lemma]
    \label{lem:K2t:EndPosition}
    For $t \geq 3$, suppose $P_2$ has built a $K^{(3)}_{2,t}(t-2)$ with main vertices $x,y$ and center $c$ in the game $\R(K_{\aleph_0}^{(3)}, G_t)$, such that
    \begin{enumerate}
        \item  $P_1$ does not have a threat;
        \item  $P_1$ has not claimed a $K_{2,t}^{(3)}$ with center $c$ and main vertex $x$;
        \item  $P_1$ has not claimed a $K_{2,t}^{(3)}$ with center $x$ and main vertex $c$.
    \end{enumerate}
    If it is $P_2$'s turn, he has a drawing strategy. \qed
\end{lemma}
We continue with proving the main theorem of this section.
\begin{theorem}
    \label{thm:K2t^3}
    $\R(K_{\aleph_0}^{(3)}, G_t)$ is a draw for all $t \geq 3$.
\end{theorem}
\begin{proof}
    Fix $t \geq 3$ and let $B$ be the infinite complete 3-uniform board $K_{\aleph_0}^{(3)}$. Let $a_1a_2a_3$ be the first edge claimed by $P_1$. In his first $2t-1$ moves, $P_2$ selects edges of the form $x_1x_2y_i$ for $i \in [2t-1]$, such that $x_1$, $x_2$ and all $y_i$ are chosen fresh at that time. Let $T_1$ be the point in time after $P_1$ has taken $2t$ edges.

    If $e^{P_1}_{T_1}(G_t) = 2t$, fix a copy $H \subseteq B$ of $G_t$ with center and main vertices $c, m_1, m_2 \in B$, such that $E_{T_1}(P_1) \subseteq E(H)$. Observe that the center and the main vertices of $H$ are uniquely determined, i.e.~there exists no other copy at $T_1$ with a center other than $c$ and main vertices other than $m_1$ and $m_2$, such that $E_{T_1}(P_1) \subseteq E(H)$. Moreover, we can assume without loss of generality, that $c = a_1$ and $m_1 = a_2$ in that case.

    We choose $x_C \in \{x_1, x_2\}$ as follows. If $e^{P_1}_{T_1}(G_t) = 2t$ and $m_2 = x_1$ or $m_2 = x_2$, we choose $x_C = x_2$ or $x_C = x_1$, respectively and therefore guarantee that $x_C$ is not a main vertex of $H$. In all other cases, we choose $x_C \in \{x_1, x_2\}$, such that $d^{P_1}_{T_1}(x_C)$ is minimal. \linebreak
    $x_C$ will be the center of the $K^{(3)}_{2,t}(t-2)$, which $P_2$ will build in order to apply the Distraction Lemma (\ref{lem:K2t:EndPosition}). We will denote the other vertex in $\{x_1, x_2\}$ with $x_M$ and it will become a main vertex of the $K^{(3)}_{2,t}(t-2)$.
    Note that, the only case in which we do not choose $d^{P_1}_{T_1}(x_C)$ to be minimal is, if $P_1$ has claimed a $(2t-1)$-star on the $a_1$-board with $x_C$ as a leaf together with $a_1x_Cx_M$. $P_2$ now picks $x_Czy_1$ for a fresh vertex $z$. Let $T_2$ be the point in time after $P_1$ has made $2t+1$ moves.
    \begin{figure}[H]
        \centering
        \begin{subfigure}[t]{0.42\textwidth}
            \centering
            \includegraphics[height= 43mm, keepaspectratio]{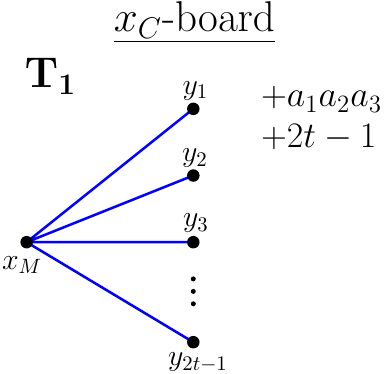}
            \caption{The $x_C$-board after $2t$ moves of $P_1$.}
            \label{pic:K2t:T1}
        \end{subfigure}
        \hspace{0.08\textwidth}
        \begin{subfigure}[t]{0.44\textwidth}
            \centering
            \includegraphics[height= 43mm, keepaspectratio]{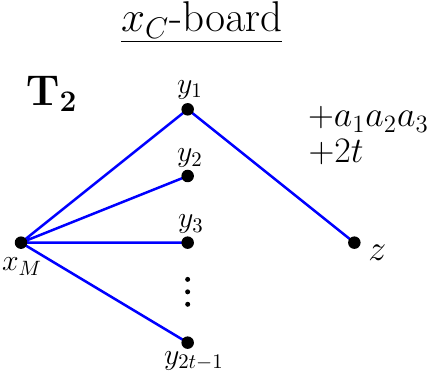}
            \caption{The $x_C$-board after $2t+1$ moves of $P_1$.}
            \label{pic:K2t:T2}
        \end{subfigure}
        \caption{}
    \end{figure}
    We obtain the following two cases.

    \underline{Case 1}: $e^{P_1}_{T_2}(G_t) < 2t+1$  \newline
    \underline{Case 2}: $e^{P_1}_{T_2}(G_t) = 2t+1$ \newline
    We will show that in either case, $P_2$ has a drawing strategy.
    \begin{proof}[Proof of Case 1]
        In his next $t-1$ moves, $P_2$ selects $t-1$ edges of the form $x_Czy_i$ for $i \in [2t-1]$, which is possible due to the pigeonhole principle. Up to relabeling, say $P_2$ has taken $x_Czy_1, \dots, x_Czy_t$ and let $T_3$ be the point in time after $P_1$ has made $3t$ moves. If $P_2$ cannot win immediately after $3t$ moves, $P_1$ must have claimed $x_Czy_{t+1}, \dots, x_Czy_{2t-1}$, since $e^{P_1}_{T_3}(G_t) < 3t$. We verify that the conditions of the Distraction Lemma (\ref{lem:K2t:EndPosition}) hold.

        \begin{figure}[H]
            \centering
            \includegraphics[width=0.30\textwidth]{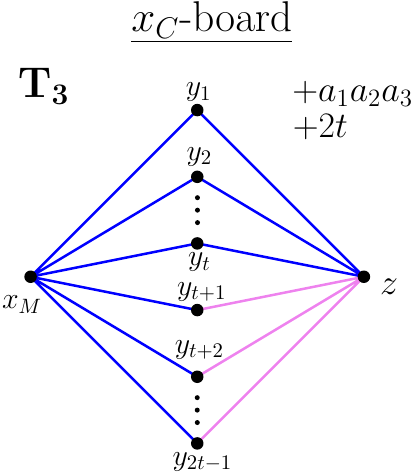}
            \caption{The $x_C$-board after $3t$ moves by $P_1$.}
            \label{pic:K2t:T3}
        \end{figure}
        If $P_1$ had a threat graph $H$ at $T_3$, then $t-1$ of the $t$ edges specified in \Cref{pic:K2t:T3} must be contained in $H$. Since $\{a_1,a_2,a_3\} \cap \{x_C, z, y_i\} = \emptyset$ for all $i \in \{t+1, \dots, 2t-1\}$, in fact all such edges $x_Czy_i$ must be contained in $H$. By choice of $z$, $x_C$ must be the center of $H$. Therefore, $d_{T_3}^{P_1}(x_C) = 3t-1$ (it cannot be larger, because $x_C \notin \{a_1, a_2, a_3\}$) and $d_{T_1}^{P_1}(x_C) = 2t-1$. By choice of $x_C$, we have $d_{T_1}^{P_1}(x_M) = 2t-1$, so $P_1$ must have claimed $2t-1$ edges of the form $x_Cx_Mv$ for $v \notin \{y_1, \dots, y_{2t-1}\}$ at $T_1$. Since $|V(G_t)| = 2t+2$, but the edges $P_1$ has claimed cover at least $3t+1$ vertices of $B$, it is clear that there cannot be a threat. In order to verify (ii) and (iii) of the Distraction Lemma (\ref{lem:K2t:EndPosition}), observe that $z$ cannot be a center of a $K_{2,t}^{(3)}$ (with main vertex $x_C$) claimed by $P_1$, because it was chosen fresh in move $2t$ by $P_2$.

        Suppose that $P_1$ has claimed a copy $H$ of $K_{2,t}^{(3)}$ at $T_3$ with center $x_C$ and main vertices $z, z_M$ for some vertex $z_M \in V(B)$. We have $x_Czy_{t+1}, \dots, x_Czy_{2t-1} \in E(H)$, since $d_{T_3}^{P_1}(z) \leq t$ and hence $x_M \neq z_M$. Let $y^*$ be the other minor vertex of $H$. By choice of $z$, we know that in her last $t$ moves, $P_1$ must have picked $x_Czy_{t+1}, \dots, x_Czy_{2t-1}, x_Czy^*$. Therefore, the edges $x_Cz_My_{t+1}, \dots, x_Cz_My_{2t-1}, x_Cz_My^*$ had already been claimed by $P_1$ at $T_1$, so $d^{P_1}_{T_1}(x_C) \geq t$. By choice of $x_C$, we have $d^{P_1}_{T_1}(x_M) \geq t$. This implies $y^* = x_M$ and the $t-1$ edges not specified at $T_1$ must contain $x_M$, but cannot contain $x_C$. Then, there can be at most two edges at $T_3$ (namely $x_Cz_My^*$ and $x_Czy^*$), which are incident to both $x_C$ and $x_M$. Applying the Distraction Lemma (\ref{lem:K2t:EndPosition}) with $x_C$ and $x_M$ as main vertex and center provides us a drawing strategy for $P_2$.
    \end{proof}
    \begin{proof}[Proof of Case 2]
        Let $H$ be a copy of a subgraph of $G_t$ with $2t+1$ edges, center $a_1$ and main vertices $a_2, a_M$, which $P_1$ has claimed at $T_2$. For convenience, let
        \vspace{10mm}
        \begin{align*}
            E(H) & = \{a_1a_2v_i, a_1a_Mv_i \ | \ i \in \{1, \dots, s_1\}\}                \\
                 & \cup  \{a_1a_2v_i \ | \ i \in \{s_1 + 1, \dots, s_1 + s'\}\}            \\
                 & \cup  \{a_1a_Mv_i \ | \ i \in \{s_1 + s' + 1, \dots, s_1 + s' +s''\}\}.
        \end{align*}
        Define $s_2 = s' + s''$. Since $2s_1 + s_2 = 2t+1$, we have that $s_2$ is odd and $s_1 \leq t$. Let
        \begin{equation*}
            A = \{a_1a_Mv_i \ | \ i \in \{s_1 + 1, \dots, s_1 + s'\}\} \cup  \{a_1a_2v_i \ | \ i \in \{s_1 + s' + 1, \dots, s_1 + s' +s''\}\}.
        \end{equation*}
        \begin{figure}[H]
            \centering
            \includegraphics[width=0.30\textwidth]{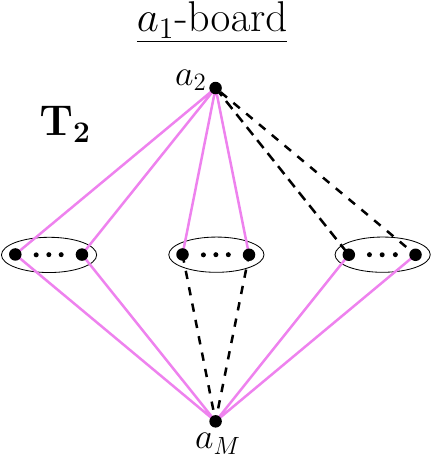}
            \caption{The $a_1$-board from the perspective of $P_1$ at $T_2$. The left blob contains the vertices $v_1, \dots, v_{s_1}$, the middle contains $v_{s_1 + 1}, \dots, v_{s_1 + s'}$ and the right contains $v_{s_1 + s'+ 1}, \dots, v_{s_1 + s' + s''}$. The dashed edges are the ones in $A$, which are claimed by neither player yet.}
            \label{pic:K2t:T2P1perspective}
        \end{figure}
        In his $(2t + 1)^\text{st}$ move, $P_2$ picks any edge from $A$. As long as he can, $P_2$ plays according to the following algorithm. We will prove later that the first time $P_2$ cannot play according to the algorithm, the Distraction Lemma (\ref{lem:K2t:EndPosition}) is applicable and he can force at least a draw.
        \begin{enumerate}
            \item Whenever $P_1$ picks an unclaimed edge of $A$, $P_2$ picks any other unclaimed edge of $A$. (This is always possible, since $A$ is even after the $(2t + 1)^\text{st}$ move of $P_2$.)
            \item If $P_1$ claims $a_1a_2v$ or $a_1a_Mv$ for $v \notin \{a_2,a_M\}$ and $a_1a_2v, a_1a_Mv \notin A$, $P_2$ chooses $a_1a_Mv$ or $a_1a_2v$ respectively.
            \item Otherwise, $P_2$ picks an unclaimed edge $x_Czy_i$ for $i \in [2t-1]$ as long as possible.
        \end{enumerate}
        Assume the game is played up to $T_4$, the point in time when $P_1$ has either won or it is $P_2$'s move and he cannot play according to the algorithm. In the latter case, all edges $x_Czy_i$ with $i \in [2t-1]$ must have been claimed at $T_4$. Let
        \begin{align*}
            E' \coloneqq E_{T_4}(P_1) \cap (\{a_1a_2v' \ | \ v' \in V(B)\backslash\{a_M\}\} \cup  \{a_1a_Mv'' \ | \ v'' \in V(B)\backslash\{a_2\}\})
        \end{align*}
        and $E_3 = E_{T_4}(P_1) \backslash E'$. Note that $E_{T_2}(P_1) \subseteq E'$. Moreover, after $P_1$ has claimed an edge $e \in E_3$, $P_2$ reacted according to (iii) in his next move, if possible. Hence, $|E_3| \leq t$ or $P_2$ has won by claiming $t+1$ edges of the form $x_Czy_i$ with $i \in [2t-1]$.

        Assume for a contradiction that $P_1$ won by building a copy $H$ of $G_t$ with center $c^*$ and main vertices $m_1^*, m_2^*$, such that $d^{P_1}_{T_4}(m_1^*) \geq 2t-1$. We prove the following claim.

        \underline{Claim}: Either the center of $H$ is not $a_1$ or one of its main vertices is not $a_2$ or $a_M$.
        \begin{claimproof}
            Assume the contrary. By the calculations above, we obtain $|A| = 2(t-s_1)+1$. Due to part (i) of the algorithm, $P_1$ can connect at most $t-s_1$ vertices from $v_{s_1+1}, \dots, v_{s_1+s_2}$ to both $a_2$ and $a_M$.  By (ii), she can connect any other vertex to at most one of $a_2, a_M$ on the $a_1$-board. So $P_1$ can obtain at most a $K_{2,t}$ on the $a_1$-board and hence cannot even claim a $K_{2,t+1}^{(3)}$ with center $a_1$ and main vertices $a_2$ and $a_M$.
        \end{claimproof}
        Which vertex can $c^*$ possibly be? We know that $d_{T_4}^{P_1}(c^*) \geq 3t$, hence the options for $c^*$ are $a_1, a_2$ and $a_M$. Both the $a_2$-board and the $a_M$-board contain an arbitrarily large star centered at $a_1$ and potentially edges from $E_3$. But since $|E_3| \leq t$, this is not enough to build a copy of $G_t$, so $c^* = a_1$. Again by a degree argument, we have $m_1^* = a_2$ without loss of generality. ($m_1^* = a_M$ is almost the same.) Since $m_2^* \neq a_M$ by the claim above, we have $|\{e \in E' \cap E(H) \ | \ m_2^* \in e\}| \leq 2$.
        Due to $|E_3| \leq t$ and $d_{T_4}^{P_1}(m_2^*) \geq t+1$, we obtain that $|\{e \in E' \cap E(H) \ | \ m_2^* \in e\}| \in \{1,2\}$, so one of $a_1a_2m_2^*, a_1a_Mm_2^*$ must be contained in $E(H)$. 
        Since $c^* m_1^*m_2^* \notin E(H)$, we obtain $a_1a_Mm_2^* \in E_{T_4}(P_1) \cap E(H)$. Moreover, all edges of $E_3$ must contain $a_1$ and $m_2^*$, because $|\{e \in E_{T_4}(P_1) \cap E(H) \ | \ a_1, m_2^* \in e\}| \leq t$, otherwise. In this case, $a_1a_Mm_2^* \notin E(H)$, since $a_1a_2a_M \notin  E_{T_4}(P_1)$ contradicting our assumption.

        So it must be $P_2$'s move and he cannot play according to the algorithm, since all edges $x_Czy_i$ with $i \in [2t-1]$ have been claimed. By the pigeonhole principle, $P_2$ has claimed the edges $x_Czy_2, \dots, x_Czy_t$ without loss of generality and if he cannot win at $T_4$, then $P_1$ must have claimed $x_Czy_{t+1}, \dots, x_Czy_{2t-1}$ in the meantime. Since $P_2$ has claimed a $K^{(3)}_{2,t}(t-2)$ and it is his move, we check the conditions of the Distraction Lemma (\ref{lem:K2t:EndPosition}).
        \begin{figure}[H]
            \centering
            \includegraphics[width=0.70\textwidth]{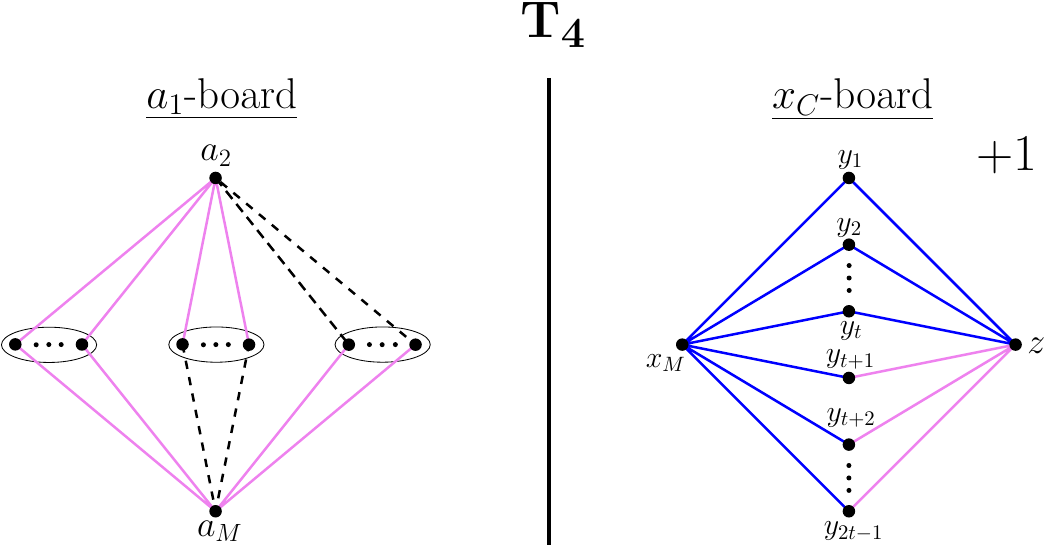}
            \caption{The actual board state at $T_4$ with the $a_1$-board on the left and the $x_C$-board on the right. Moreover, $P_1$ has claimed at most $(t-s_1)$ edges of $A$ (the dashed edges as in \Cref{pic:K2t:T2P1perspective}). Since $|E_3| = t$, note that all $P_1$-edges except one are displayed either in the left part or the right part of the figure.}
            \label{pic:K2t:T4}
        \end{figure}

        Since $e^{P_1}_{T_2}(G_t) = 2t+1$, we had $e^{P_1}_{T_1}(G_t) = 2t$. Moreover, $a_M$ has been chosen by $P_1$ at latest in her $(2t)^\text{th}$ move, so by choice of $x_C$, we have $x_C \notin \{a_1, a_2, a_M\}$. Since $z$ was chosen in move $2t$, we also have $z \notin \{a_1, a_2, a_M\}$. So $d^{P_1}_{T_4}(x_C), d^{P_1}_{T_4}(z) \leq t+2$ and since $t \geq 3$, there cannot be a $K_{2,t}^{(3)}$ claimed by $P_1$, with center $x_C$ or $z$ and hence (ii) and (iii) of the Distraction Lemma (\ref{lem:K2t:EndPosition}) hold.

        If $P_1$ had a threat graph $H$ at $T_4$, then its center must be contained in at least $3t-1$ edges. By the same degree argument as above, the only possible choices for its center are $a_1, a_2, a_M$. The $a_2$-board contains a $P_1$-star centered at $a_1$ and at most one additional $P_1$-edge, since $x_Czy_{t+1}, \dots, x_Czy_{2t-1} \in E_3$. Since it could be that $a_M = y_i$ for some $y_i$, the \linebreak $a_M$-board contains at most a $P_1$-star and two additional $P_1$-edges, so both $a_2$ and $a_M$ cannot be the centers of $H$. On the $a_1$-board, only $a_2$ and $a_M$ have degree at least 4 at $T_4$, so if $t \geq 4$, only $a_2$ and $a_M$ could be the main vertices of $H$. By definition of $T_4$ the last edge $P_1$ claimed before $T_4$ must be contained in $H$, but cannot be of the form $a_1a_2v$ or $a_1a_Mv$ for some $v$, because in that case $P_2$ reacted with part (i) or (ii) of the algorithm. Therefore, $P_1$ cannot have a threat graph at $T_4$, if $t \geq 4$.

        If $t=3$, the only type of vertex $v \notin \{a_2, a_M\}$, which satisfies $d^{P_1}_{T_4}(v) = 3$ on the $a_1$-board is connected with $a_2$ and $a_M$ and contained in the non-specified edge from $E_3$. Since either $a_2$ or $a_M$ must be the other main vertex of $H$, $a_1a_2v \notin E(H)$ or $a_1a_Mv \notin E(H)$ respectively. Therefore, $|E(H) \cap \{e \in E_{T_4}(P_1) \ | \ v \in e \}| \leq 2$ contradicting that $v$ was a main vertex of $H$.
    \end{proof}
    This shows that $P_2$ has also a drawing strategy in \textit{Case 2}, which completes the proof.
\end{proof}

\printbibliography

\end{document}